\documentclass[10pt]{article}

\usepackage{amssymb, latexsym, amsmath}
\usepackage{amsfonts}
\usepackage{amsthm}
\usepackage[vmargin={1.25in,1.25in},
            hmargin={1.3in,1.3in}
            ]{geometry}
\usepackage{authblk}

\theoremstyle{definition}
\newtheorem{theorem}{Theorem}
\newtheorem{lemma}[theorem]{Lemma}

\newtheorem{definition}[theorem]{Definition}

\newcommand{\closure}[1]{\left\langle #1 \right\rangle}
\newcommand{\B}{\operatorname{B}}
\newcommand{\Bin}{\operatorname{Bin}}

\title{Lower bounds for bootstrap percolation on
  {G}alton--{W}atson trees}

\author{Karen Gunderson\thanks{Heilbronn Institute for Mathematical Research, School of Mathematics, University of Bristol, Bristol BS8 1TW, UK.} \:and Micha{\l} Przykucki\thanks{Department of Pure Mathematics and Mathematical Statistics, University of Cambridge, Wilberforce Road, Cambridge CB3 0WB, UK, and London Institute for Mathematical Sciences, 35a South St, Mayfair, London W1K 2XF, UK. Supported in part by MULTIPLEX no.\ 317532.}}

\date{12 February 2014}

\begin{document}

\maketitle

\begin{abstract}
\small
Bootstrap percolation is a cellular automaton modelling the spread of an `infection' on a graph. In this note, we prove a family of lower bounds on the critical probability for $r$-neighbour bootstrap percolation on Galton--Watson trees in terms of moments of the offspring distributions. With this result we confirm a conjecture of Bollob{\'a}s, Gunderson, Holmgren, Janson and Przykucki. We also show that these bounds are best possible up to positive constants not depending on the offspring distribution.
\end{abstract}

\small



\normalsize

\section{Introduction}

Bootstrap percolation, a type of cellular automaton, was introduced by Chalupa, Leath and Reich \cite{bootstrapbethe} and has been used to model a number of physical processes.  Given a graph $G$ and threshold $r \geq 2$, the \emph{$r$-neighbour bootstrap process} on $G$ is defined as follows: Given $A \subseteq V(G)$, set $A_0 = A$ and for each $t \geq 1$, define
\[
A_t = A_{t-1} \cup \{v \in V(G) :\ |N(v) \cap A_{t-1}| \geq r\},
\] 
where $N(v)$ is the neighbourhood of $v$ in $G$.  The closure of a set $A$ is $\closure{A} = \bigcup_{t \geq 0} A_t$.  Often the bootstrap process is thought of as the spread, in discrete time steps, of an `infection' on a graph.  Vertices are in one of two states: `infected' or `healthy' and a vertex with at least $r$ infected neighbours becomes itself infected, if it was not already, at the next time step.  For each $t$, the set $A_t$ is the set of infected vertices at time $t$.   A set $A \subseteq V(G)$ of initially infected vertices is said to \emph{percolate} if $\closure{A} = V(G)$.

Usually, the behaviour of bootstrap processes is studied in the case where the initially infected vertices, i.e., the set $A$, are chosen independently at random with a fixed probability $p$.  For an infinite graph $G$ the \emph{critical probability} is defined by
\begin{equation*}
p_c(G, r) = \inf\{p :\ \mathbb{P}_p(\closure{A} = V(G))>0\}.
\end{equation*} 
This is different from the usual definition of critical probability for finite graphs, which is generally defined as the infimum of the values of $p$ for which percolation is more likely to occur than not.

In this paper, we consider bootstrap percolation on Galton--Watson trees and answer a conjecture in \cite{bootsGW} on lower bounds for their critical probabilities. For any offspring distribution $\xi$ on $\mathbb{N} \cup \{0\}$, let $T_{\xi}$ denote a random Galton--Watson tree with offspring distribution $\xi$.  For any fixed offspring distribution $\xi$, the critical probability $p_c(T_\xi, r)$ is almost surely a constant (see Lemma 3.2 in \cite{bootsGW}) and we shall give lower bounds on the critical probability in terms of various moments of $\xi$.

Bootstrap processes on infinite regular trees were first considered by Chalupa, Leath and Reich \cite{bootstrapbethe}. Later,  Balogh, Peres and Pete \cite{infiniteTrees} studied bootstrap percolation on arbitrary infinite trees and one particular example of a random tree given by a Galton--Watson branching process.  In \cite{bootsGW}, Galton--Watson branching processes were further considered, and it was shown that for every $r \geq 2$, there is a constant $c_r > 0$ so that 
\begin{equation*}
p_c(T_\xi, r) \geq \frac{c_r}{\mathbb{E}[\xi]}\exp\left(-\frac{\mathbb{E}[\xi]}{r-1} \right)
\end{equation*}
and in addition, for every $\alpha \in (0,1]$, there is a positive constant $c_{r, \alpha}$ so that, 
\begin{equation}\label{eq:alpha}
p_c(T_\xi, r) \geq c_{r, \alpha} \left(\mathbb{E}[\xi^{1+\alpha}] \right)^{-1/\alpha}.
\end{equation}
Additionally, in \cite{bootsGW} it was conjectured that for any $r \geq 2$, inequality \eqref{eq:alpha} holds for any $\alpha \in (0, r-1]$. As our main result, we show that this conjecture is true. For the proofs to come, some notation from \cite{bootsGW} is used. If an offspring distribution $\xi$ is such that $\mathbb{P}(\xi < r) > 0$, then one can easily show that $p_c(T_\xi, r) = 1$. With this in mind, for $r$-neighbour bootstrap percolation, we only consider offspring distributions with $\xi \geq r$ almost surely.

\begin{definition}\label{def:G}
For every $r \geq 2$ and $k \geq r$, define
\[
g_k^r(x) = \frac{\mathbb{P}(\Bin(k, 1-x) \leq r-1)}{x} = \sum_{i=0}^{r-1} \binom{k}{i} x^{k-i-1}(1-x)^i
\]
and for any offspring distribution $\xi$ with $\xi \geq r$ almost surely, define
\[
G_{\xi}^r(x) = \sum_{k \geq r} \mathbb{P}(\xi = k) g_k^r(x).
\]
\end{definition} 

Some facts, which can be proved by induction, about these functions are used in the proofs to come. For any $r \geq 2$, we have $g_r^r(x) = \sum_{i=0}^{r-1} (1-x)^i$ and for any $k > r$,
\begin{equation}\label{eq:gk_rec}
g_r^r(x) - g_k^r(x) = \sum_{i=r}^{k-1} \binom{i}{r-1} x^{i-r}(1-x)^r.
\end{equation}
Hence, for all distributions $\xi$ we have $G_{\xi}^r(x) \leq g_r^r(x)$ for $x \in [0,1]$. 

Developing a formulation given by Balogh, Peres and Pete \cite{infiniteTrees}, it was shown in \cite{bootsGW} (see Theorem 3.6 in \cite{bootsGW}) that if $\xi \geq r$, then
\begin{equation}\label{eq:pc_M}
p_c(T_\xi, r) = 1 - \frac{1}{\max_{x \in [0,1]} G_{\xi}^r(x)}.
\end{equation}

\section{Results}

In this section, we shall prove a family of lower bounds on the critical probability $p_c(T_{\xi}, r)$ based on the $(1+\alpha)$-moments of the offspring distributions $\xi$ for all $\alpha \in (0,r-1]$, using a modification of the proofs of Lemmas 3.7 and 3.8 in \cite{bootsGW} together with some properties of the gamma function and the beta function.  

Recall that the gamma function is given, for $z$ with $\Re(z) > 0$, by $\Gamma(z) = \int_0^{\infty} t^{z-1} e^{-t}\ dt$ and for all $n \in \mathbb{N}$, satisfies $\Gamma(n) = (n-1)!$.  The beta function is given, for $\Re(x), \Re(y)>0$, by $\B(x,y) = \int_0^1 t^{x-1}(1-t)^{y-1}\ dt$ and satisfies $\B(x,y) = \frac{\Gamma(x)\Gamma(y)}{\Gamma(x+y)}$. We shall use the following bounds on the ratio of two values of the gamma function obtained by Gautschi \cite{gammainequalities}. For $n \in \mathbb{N}$ and $0 \leq s \leq 1$ we have
\begin{equation}
\label{eq:gammasIneq}
 \left( \frac{1}{n+1} \right)^{1-s} \leq \frac{\Gamma(n+s)}{\Gamma(n+1)} \leq \left( \frac{1}{n} \right)^{1-s}.
\end{equation}
Let us now state our main result.

\begin{theorem}\label{thm:main}
 For each $r \geq 2$ and $\alpha \in (0, r-1]$, there exists a constant $c_{r, \alpha}>0$ such that for any offspring distribution $\xi$ with $\mathbb{E}[\xi^{1+\alpha}] < \infty$, we have
\[
 p_c(T_{\xi},r) \geq c_{r,\alpha} \left( \mathbb{E} \left [\xi^{1+\alpha} \right ] \right)^{-1/\alpha}.
\]
\end{theorem}

We prove Theorem \ref{thm:main} in two steps. First, in Lemma \ref{lem:openInterval}, we show that it holds for $\alpha \in (0, r-1)$. Then, in Lemma \ref{lem:r-1}, we consider the case $\alpha = r-1$.

\begin{lemma}\label{lem:openInterval}
 For all $r \geq 2$ and $\alpha \in (0, r-1)$, there exists a positive constant $c_{r, \alpha}$ such that for any distribution $\xi$ with $\mathbb{E}[\xi^{1+\alpha}] < \infty$, we have
\[
 p_c(T_{\xi},r) \geq c_{r,\alpha} \left( \mathbb{E} \left [\xi^{1+\alpha} \right ] \right)^{-1/\alpha}.
\]
\end{lemma}

\begin{proof}
 Fix $r \geq 2$, $\alpha \in (0, r-1)$ with $\alpha \notin \mathbb{Z}$ and an offspring distribution $\xi$.  Set $t = \lfloor \alpha \rfloor$ and $\varepsilon = \alpha -t$ so that $\varepsilon \in (0,1)$ and $t$ is an integer with $t \in [0, r-2]$.  Set $M = \max_{x \in [0,1]}G_{\xi}^r(x)$ and fix $y \in [0,1]$ with the property that $g_r^r(1-y) = M$. Such a $y$ can always be found since $G_{\xi}^r(x) \leq g_r^r(x)$ in $[0,1]$, $G_{\xi}^r(1) = g_r^r(1) = 1$ and $g_r^r(x)$ is continuous.  Thus, $M = 1+y+\ldots+y^{r-1}$ and so by equation \eqref{eq:pc_M}
\begin{equation}
\label{eq:p_cBoundx}
 p_c(T_\xi,r) = 1-\frac{1}{M} = \frac{y(1-y^{r-1})}{1-y^r} \geq \frac{r-1}{r} y.
\end{equation}
A lower bound on $p_c(T_\xi,r)$ is given by considering upper and lower bounds for the integral $\int_0^1 \frac{g_r^r(x) - G_\xi^r(x)}{(1-x)^{2+\alpha}}\ dx$.

For the upper bound, using the definition of the beta function, for every $k \geq r$
\begin{align}
\int_0^1 \frac{g_r^r(x) - g_k^r(x)}{(1-x)^{\alpha+2}}\ dx 
		&=\sum_{i=r}^{k-1} \binom{i}{r-1} \int_0^1 x^{i-r} (1-x)^{r-2 - \alpha}\ dx \qquad \text{(by eq. \eqref{eq:gk_rec})}\notag \\
		&=\sum_{i=r}^{k-1} \binom{i}{r-1} B(i-r+1, r-1-\alpha) \notag\\
		&=\sum_{i=r}^{k-1} \frac{i!}{(r-1)!(i-r+1)!} \frac{(i-r)! \Gamma(r-1-\alpha)}{\Gamma(i-\alpha)} \notag\\
		&=\sum_{i=r}^{k-1} \frac{i(i-1)\ldots (i-t) \Gamma(i-t)}{(i-r+1)\Gamma(i-t-\varepsilon)} \notag \\
		& \qquad \cdot \frac{\Gamma(r-1-t-\varepsilon)}{(r-1)(r-2)\ldots (r-1-t)\Gamma(r-1-t)}. \label{eq:int_ub1}
\end{align}
Let $c_1 = c_1(r, \alpha) = \frac{\Gamma(r-1-t-\varepsilon)}{(r-1)(r-2)\ldots (r-1-t)\Gamma(r-1-t)}$.  Note that by inequality \eqref{eq:gammasIneq}, for $t< r-2$, $\frac{\Gamma(r-1-t-\varepsilon)}{\Gamma(r-1-t)} \geq \frac{1}{(r-1-t)^{\varepsilon}}$ and so $c_1 \geq \frac{1}{(r-1)^{t+\varepsilon}} = (r-1)^{-\alpha}$.  On the other hand, if $t=r-2$, then $c_1 = \frac{\Gamma(1-\varepsilon)}{(r-1)!} = \frac{\Gamma(2-\varepsilon)}{(1-\varepsilon)(r-1)!} \geq \frac{1}{2(r-1)!(1-\varepsilon)}$.   

Thus, continuing equation \eqref{eq:int_ub1}, applying inequality \eqref{eq:gammasIneq} again yields
\begin{align*}
\sum_{i=r}^{k-1} &\frac{i(i-1)\ldots (i-t) \Gamma(i-t)}{(i-r+1)\Gamma(i-t-\varepsilon)} \cdot \frac{\Gamma(r-1-t-\varepsilon)}{(r-1)(r-2)\ldots (r-1-t)\Gamma(r-1-t)}\\
	&\leq c_1 \sum_{i=r}^{k-1} \frac{i}{i-r+1} (i-1)(i-2)\ldots(i-t) (i-t)^{\varepsilon}\\
	&\leq r c_1 \sum_{i=r}^{k-1} i^{t+\varepsilon}\\
	&\leq r c_1 k^{1+t+\varepsilon} = r c_1 k^{1+\alpha}.
\end{align*}
Thus, taking expectation over $k$ with respect to $\xi$,
\begin{equation}\label{eq:int_ub_tot}
\int_0^1 \frac{g_r^r(x) - G_{\xi}^r(x)}{(1-x)^{2+\alpha}}\ dx \leq r c_1 \mathbb{E}[\xi^{1+\alpha}].
\end{equation}

Consider now a lower bound on the integral:
\begin{align*}
\int_0^1 &\frac{g_r^r(x) - G_\xi^r(x)}{(1-x)^{2+\alpha}}\ dx
		\geq \int_0^{1-y} \frac{g_r^r(x) - M}{(1-x)^{2+\alpha}}\ dx\\
		&=\int_0^{1-y} -\frac{(M-1)}{(1-x)^{2+\alpha}} + \sum_{i=0}^{r-2} \frac{1}{(1-x)^{1+\alpha-i}}\ dx\\
		&=\left[-\frac{(M-1)}{(\alpha+1)(1-x)^{1+\alpha}} +\sum_{i=0}^{r-2} \frac{1}{(\alpha-i)(1-x)^{\alpha-i}} \right]_0^{1-y}\\
		&=-\frac{(M-1)}{(\alpha+1)}\left(\frac{1}{y^{1+\alpha}}-1 \right) + \sum_{i=0}^t \frac{1}{\alpha-i}\left(\frac{1}{y^{\alpha-i}} - 1\right) + \sum_{i=t+1}^{r-2} \frac{1-y^{i-\alpha}}{i-\alpha}\\
		&=\frac{1}{y^{\alpha}} \left(\frac{M-1}{\alpha+1}\left(\frac{y^{\alpha+1}-1}{y} \right) + \sum_{i=0}^t \frac{y^i - y^\alpha}{\alpha-i} + \sum_{i=t+1}^{r-2} \frac{y^\alpha-y^i}{i-\alpha} \right)\\
		&=\frac{1}{y^\alpha} \left( \frac{(1+y+y^2 + \ldots + y^{r-2})(y^{\alpha+1}-1)}{(\alpha+1)} +\sum_{i=0}^t \frac{y^i - y^\alpha}{\alpha-i} + \sum_{i=t+1}^{r-2} \frac{y^\alpha-y^i}{i-\alpha}\right)\\
		&=\frac{1}{y^\alpha}\Bigg(\frac{-1}{\alpha+1} + \frac{1}{\alpha} + \sum_{i=1}^t \left(\frac{y^i}{\alpha-i} - \frac{y^i}{\alpha+1} \right) +\sum_{i=0}^{r-2} \frac{y^{\alpha+1+i}}{\alpha+1} - \sum_{i=t+1}^{r-2} \frac{y^i}{\alpha+1} - \sum_{i=0}^t \frac{y^{\alpha}}{\alpha-i} \\
		& \qquad + \sum_{i=t+1}^{r-2} \frac{y^\alpha-y^i}{i-\alpha}\Bigg)\\
		&\geq \frac{1}{y^{\alpha}}\left(\frac{1}{\alpha(\alpha+1)} - \frac{y^{t+1}}{\alpha+1} - \sum_{i=0}^t \frac{y^\alpha}{\alpha-i} \right)\\
		&\geq \frac{1}{y^{\alpha}}\left(\frac{1}{\alpha(\alpha+1)} - y^{\alpha} \sum_{i=0}^{t+1} \frac{1}{\alpha+1-i}\right).
\end{align*}

Set $c_2 = c_2(\alpha) = \sum_{i=0}^{t+1} \frac{1}{\alpha+1-i}$ and consider separately two different cases.  For the first, if $y^{\alpha} c_2 \geq \frac{1}{2\alpha(\alpha+1)}$ then since $\mathbb{E}[\xi^{\alpha+1}] \geq 1$,
\[
y^\alpha \geq \frac{1}{2\alpha(\alpha+1)c_2} \geq \frac{1}{2\alpha(\alpha+1)c_2} \mathbb{E}[\xi^{1+\alpha}]^{-1}.
\]
Thus, if $c_2' = \left(\frac{1}{2\alpha(\alpha+1)c_2} \right)^{1/\alpha}$, then $y \geq c_2' \mathbb{E}[\xi^{1+\alpha}]^{-1/\alpha}$.

In the second case, if $y^\alpha < \frac{1}{2\alpha(\alpha+1)c_2}$, then
\begin{equation}\label{eq:int_lb}
\int_0^1 \frac{g_r^r(x) - G_\xi^r(x)}{(1-x)^{2+\alpha}}\ dx \geq \frac{1}{y^\alpha} \frac{1}{2\alpha(\alpha+1)}.
\end{equation}
Combining equation \eqref{eq:int_lb} with equation \eqref{eq:int_ub_tot} yields
\[
y^\alpha \geq \frac{1}{2\alpha(\alpha+1)} \frac{1}{r c_1} \mathbb{E}[\xi^{1+\alpha}]^{-1}
\]
and setting $c_1' = (2\alpha(\alpha+1) r c_1)^{-1/\alpha}$ gives $y \geq c_1' \mathbb{E}[\xi^{1+\alpha}]^{-1/\alpha}$.

Finally, set $c_{r, \alpha} = \frac{r-1}{r} \min\{c_1', c_2'\}$ so that by inequality \eqref{eq:p_cBoundx} we obtain,
\[
p_c(T_\xi, r) \geq \frac{r-1}{r} y \geq c_{r, \alpha} \mathbb{E}[\xi^{1+\alpha}]^{-1/\alpha}.
\]

For every natural number $n \in [1, r-2]$, note that $\lim_{\alpha \to n^-} c_{r, \alpha} >0$ and, by the monotone convergence theorem, there is a constant $c_{r,n}>0$ so that 
\[
p_c(T_\xi,r) \geq c_{r,n} \mathbb{E}[\xi^{1+n}]^{-1/n}.
\]
This completes the proof of the lemma.
\end{proof}

In the above proof, as $\alpha \to (r-1)^-$, $c_1(r, \alpha) \to \infty$ and hence $\lim_{\alpha \to (r-1)^-} c_{r,\alpha} = 0$, so the proof of Lemma \ref{lem:openInterval} does not directly extend to the case $\alpha = r-1$. We deal with this problem in the next lemma. Using a different approach we prove an essentially best possible lower bound on $p_c(T_\xi, r)$ based on the $r$-th moment of the distribution $\xi$. The sharpness of our bound is demonstrated by the $b$-branching tree $T_b$, a Galton--Watson tree with a constant offspring distribution, for which, as a function of $b$, we have $p_c(T_b, r) = (1+o(1))(1-1/r)\left(\frac{(r-1)!}{b^r} \right)^{1/(r-1)}$ (see Lemma 3.7 in \cite{bootsGW}).

\begin{lemma}\label{lem:r-1}
For any $r \geq 2$ and any offspring distribution $\xi$ with $\mathbb{E}[\xi^r] < \infty$,
\[
p_c(T_\xi, r) \geq \left(1 - \frac{1}{r}\right) \left( \frac{(r-1)!}{\mathbb{E}[\xi^r]}\right)^{1/(r-1)}.
\]
\end{lemma}

\begin{proof}
As in the proof of Lemma 3.7 of \cite{bootsGW} note that for every $k \geq r$ and $t \in [0,1]$,
\begin{align}
	g_k^r(1-t)
		&=\frac{\mathbb{P}(\Bin(k, t) \leq r-1)}{1-t} = \frac{1-\mathbb{P}(\Bin(k,t) \geq r)}{1-t} \notag\\
		&\geq \frac{1-\binom{k}{r} t^r}{1-t} \geq \frac{1-\frac{1}{r!}k^r t^r}{1-t} \label{eq:lb_gk}.
\end{align}

Using the lower bound in inequality \eqref{eq:lb_gk} for the function $G_\xi^r(x)$ yields
\begin{equation*}
G_{\xi}^r(1-t) \geq \sum_{k\geq r} \mathbb{P}(\xi = k) \frac{1-\frac{1}{r!}k^r t^r}{1-t} = \frac{1 - \frac{t^r}{r!}\mathbb{E}[\xi^r]}{1-t}.
\end{equation*}
Evaluating the function $G_\xi^r(1-t)$ at $t = t_0 = \left(\frac{(r-1)!}{\mathbb{E}[\xi^r]} \right)^{1/(r-1)}$ yields
\[
G_\xi^r(1-t_0) \geq \frac{1-\frac{t_0^r}{r!}\mathbb{E}[\xi^r]}{1-t_0} = \frac{1-\frac{1}{r}t_0}{1-t_0}.
\]
Since the maximum value of $G_{\xi}^r(x)$ is at least as big as $G_\xi^r(1-t_0)$, by equation \eqref{eq:pc_M},
\begin{align*}
p_c(T_\xi, r)
	&\geq 1 - \frac{1}{G_\xi^r(1-t_0)} = \frac{G_\xi^r(1-t_0) - 1}{G_\xi^r(1-t_0)} \\
	&=\frac{t_0 \left(1-\frac{1}{r}\right)}{1-t_0} \frac{1-t_0}{1-\frac{1}{r}t_0} \\
	& = \frac{t_0 \left(1-\frac{1}{r}\right)}{1-t_0/r} \geq t_0 \left(1- \frac{1}{r}\right) \\
	&= \left(1-\frac{1}{r}\right) \left(\frac{(r-1)!}{\mathbb{E}[\xi^r]} \right)^{1/(r-1)}.
\end{align*}
This completes the proof of the lemma.
\end{proof}

Theorem \ref{thm:main} now follows immediately from Lemmas \ref{lem:openInterval} and \ref{lem:r-1}.

It is not possible to extend a result of the form of Theorem \ref{thm:main} to $\alpha > r-1$, as demonstrated, again, by the regular $b$-branching tree. For every $\alpha$, the $(1+\alpha)$-th moment of this distribution is $b^{1+\alpha}$ and the critical probability for the constant distribution is $p_c(T_b, r) = (1+o(1))(1-1/r)\left(\frac{(r-1)!}{b^r} \right)^{1/(r-1)}$.

As we already noted, Lemma \ref{lem:r-1} is asymptotically sharp, giving the best possible constant in Theorem \ref{thm:main} for any $r \geq 2$ and $\alpha = r-1$. We now show that for $\alpha \in (0, r-1)$, Theorem \ref{thm:main} is also best possible, up to constants.  In \cite{bootsGW}, it was shown that for every $r \geq 2$, there is a constant $C_r$ such that if $b \geq (r-1)\log(4er)$, then there is an offspring distribution $\eta_{r,b}$ with $\mathbb{E}[\eta_{r,b}] = b$ and $p_c(T_{\eta_{r,b}}, r) \leq C_r e^{-\frac{b}{r-1}}$.  It was shown that there are $k_1 = k_1(r,b) \leq e(r-2)e^{\frac{b}{r-1}}-1$ and $A, \lambda \in (0,1)$ so that the distribution $\eta_{r,b}$ is given by
\begin{equation*}
\mathbb{P}(\eta_{r,b} = k) = 
	\begin{cases}
		\frac{r-1}{k(k-1)}	&r< k \leq k_1, k \neq 2r+1\\
		\frac{1}{r} + \lambda A	&k = r\\
		\frac{r-1}{(2r+1)2r} + (1-\lambda)A	&k = 2r+1.
	\end{cases}
\end{equation*}

For any $\alpha >0$, the $(\alpha+1)$-th moment of $\eta_{r, b}$ is bounded from above as follows,
\begin{align*}
\mathbb{E}[\eta_{r,b}^{\alpha+1}]
	&=\sum_{k=r}^{k_1} \frac{(r-1)}{k(k-1)}k^{\alpha+1} + \lambda A r^{\alpha + 1} + (1-\lambda) A (2r+1)^{\alpha+1}\\
	&\leq 2(r-1)\sum_{k=r}^{k_1} k^{\alpha-1} + 2(2r+1)^{\alpha+1}\\
	&\leq 2(r-1) \left ( \int_r^{k_1+1} x^{\alpha-1}\ dx + r^{\alpha-1} \right ) + 2(2r+1)^{\alpha+1}\\
	&\leq \frac{2(r-1)}{\alpha} (k_1+1)^{\alpha} + 3(2r+1)^{\alpha+1}\\
	&\leq \frac{2(r-1)}{\alpha} \left(e(r-2) e^{\frac{b}{r-1}}\right)^{\alpha} + 3(2r+1)^{\alpha+1},
\end{align*}
where the $r^{\alpha-1}$ term makes the inequality hold for $\alpha < 1$. In particular, there is a constant $C_{r, \alpha}$ so that for $b$ sufficiently large, $\mathbb{E}[\eta_{r,b}^{1+\alpha}]^{1/\alpha} \leq C_{r, \alpha} e^{\frac{b}{r-1}}$.  Thus, for some positive constant $C_{r, \alpha}'$, 
\[
p_c(T_{\eta_{r,b}}, r) \leq C_r e^{-\frac{b}{r-1}} \leq C_{r, \alpha}' \mathbb{E}[\eta_{r,b}^{1+\alpha}]^{-1/\alpha}.
\]
Hence the bounds in Theorem \ref{thm:main} are sharp up to a constant that does not depend on the offspring distribution $\xi$.

\end{document}